\documentclass[11pt,letterpaper]{amsart}

\usepackage{graphicx, amsfonts, amssymb}
\usepackage{mathrsfs, amsmath, amsthm}
\usepackage{verbatim}
\usepackage{relsize}
\usepackage{enumerate}
\usepackage{hyperref}
\usepackage{version}
\usepackage{blindtext}
\usepackage{mathtools}
\usepackage{color}
\usepackage{ stmaryrd }
\usepackage{ esint }
\usepackage{bigints}
\usepackage{bbm}

\newtheorem{theorem}{Theorem}

\newtheorem{Lemm}[theorem]{Lemma}
\newtheorem{prop}[theorem]{Proposition}
\newtheorem{corollary}[theorem]{Corollary}

\theoremstyle{remark}

\title[Spectrum of Hausdorff operators]{Spectrum of Hausdorff operators on weighted Bergman and Hardy spaces of the upper half-plane}

\author[C. Bellavita]{Carlo Bellavita}
\email{carlo.bellavita@gmail.com}
\email{carlo.bellavita@unimi.it}
\address{Dipartimento di Matematica ``F. Enriques''\\
  Dipartimento di Eccellenza MUR 2023-2027\\
Universit\`a degli Studi di Milano\\
Via C. Saldini 50\\
20133 Milano, Italy.}

\author[G. Stylogiannis]{Georgios Stylogiannis}
\email{g.stylog@gmail.com}
\email{stylog@math.auth.gr}
\address{Department of Mathematics\\
Aristotle University of Thessaloniki\\
\\54124 Thessaloniki, Greece.}

\thanks{The first author is member of the Gruppo Nazionale per l’Analisi Matematica, la Probabilità e le loro Applicazioni (GNAMPA) of the Istituto Nazionale di Alta Matematica (INdAM)}


\begin{document}

\begin{abstract}
We characterize the spectrum of Hausdorff operators on weighted Bergman and power weighted Hardy spaces of the upper half-plane.
\end{abstract}

\maketitle

\section*{Introduction}

The study of integral operators on function spaces is a cornerstone of concrete operator theory. Among these, the {Hausdorff operator} $\mathcal{H}_{\phi}$ is of particular interest due to its relationship with other fundamental operators, such as the Cesàro operator, and its applications in various branches of analysis \cite{Hausdorff1921, Liflyand1999}.
For a holomorphic function $f$ in the upper half-plane $\mathbb{U}$, the Hausdorff operator is defined by the integral
\begin{equation*}
\mathcal{H}_{\phi}f(z)=\int_{0}^{\infty}f\left(\frac{z}{t}\right)\frac{\phi(t)}{t}dt, \quad z\in\mathbb{U}
\end{equation*}
where $\phi$ is a measurable kernel. If $f$ is a measurable function on the real line $\mathbb{R}$, the operator takes a similar form 
\[
H_{\phi}f(x)=\int_{0}^{\infty}f\left( \frac{x}{t}\right) \frac{\phi(t)}{t}dt \quad x\in\mathbb{\mathbb{R}}.
\]
The modern theory of Hausdorff operator branches into two primary directions: the \emph{complex analysis setting} established by Siskakis and Galanopoulos \cite{Galanopoulos2001, Petrosgeorgios2023, Siskakis1987, Siskakis1990, Siskakis1996, Georgios2020}, and the \emph{Lebesgue spaces/ Fourier transform setting} introduced by Georgakis \cite{Georgakis1992}, Liflyand and Móricz \cite{Liflyand1999}. In this article, we aim to integrate and use both these two perspectives.

\medskip
The focus of this article is to characterize the {spectrum} of $\mathcal{H}_{\phi}$ on two spaces of analytic functions within the upper half-plane, the power weighted Hardy spaces $H_{|\cdot|^{a}}^{p}(\mathbb{U})$ and the weighted Bergman spaces $A_{a}^{p}(\mathbb{U})$. We recall that, given the Banach space $X$, the {spectrum} of $\mathcal{H}_{\phi}$ is made by all the complex numbers $\lambda$ such that the operator $\lambda I - \mathcal{H}_{\phi}$ is not boundedly invertible on $X$, where $I$ is the identity operator. The spectrum is denoted by $\sigma(\mathcal{H}_{\phi},X)$ or, more shortly, $\sigma(\mathcal{H}_{\phi})$ when the space $X$ does not need to be specified.

\begin{theorem}\label{Theorem spectrum weighted Hardy}
Let $1\leq p<\infty, a>-1$ and $\phi$ be a measurable function in $[0,\infty)$ which satisfies
$$
\int_{0}^{\infty} |\phi(t)|t^{\frac{a+1}{p}-1}\,dt<\infty.
$$
Then 
$
\sigma(\mathcal{H}_{\phi}, H^{p}_{|\cdot|^{a}}(\mathbb{U}))=\overline{\widehat{k_{a,p}}(\mathbb{R})}
$
where $\widehat{k_{a,p}}(\xi)=\int^{\infty}_{0}\phi(t)t^{\frac{a+1}{p}}\frac{dt}{t^{1+i\xi}}$ for $\xi\in\mathbb{R}$.
\end{theorem}

\begin{theorem}\label{Theorem spectrum Bergman}
Let $1\leq p<\infty, a>0$ and $\phi$ be a measurable function in $[0,\infty)$ which satisfies
$$
\int_{0}^{\infty} |\phi(t)|t^{\frac{a+1}{p}-1}\,dt<\infty.
$$
Then 
$
\sigma(\mathcal{H}_{\phi}, A^{p}_{{a}}(\mathbb{U}))=\overline{\widehat{k_{a,p}}(\mathbb{R})}
$
where $\widehat{k_{a,p}}(\xi)=\int^{\infty}_{0}\phi(t)t^{\frac{a+1}{p}}\frac{dt}{t^{1+i\xi}}$ with $\xi\in\mathbb{R}$.
\end{theorem}
\noindent Theorems \ref{Theorem spectrum weighted Hardy} and \ref{Theorem spectrum Bergman} state that the spectrum of the Hausdorff operator $\mathcal{H}_\phi$ in both the Hardy and Bergman spaces coincides with the closure of the image of the real line under a specific Fourier-type transform of the kernel $\phi$.

By establishing these characterizations, we obtain new lower bounds for the norm of the Hausdorff operator and we extend results regarding the spectral behavior of the Cesàro-like operators
$$
C_{\nu}f(z)=\frac{1}{z^{\nu}}\int_{0}^{z}\zeta^{\nu-1}f(\zeta)\,d\zeta.
$$
\begin{corollary}\label{last corollary}
\

\noindent (i) The Cesaro operator $\mathcal{C}_{\nu}$ is bounded on $H^{p}_{|\cdot|^{a}}(\mathbb{U})$ for $p \cdot  \mbox{Re} \nu>a+1$
with $\|C_{\nu}\|_{H^{p}_{|\cdot|^{a}}}=p(\mbox{Re}\nu-a-1)^{-1}$ and
\begin{align*}
\sigma(\mathcal{C}_{\nu},H^{p}_{|\cdot|^{a}}(\mathbb{U}))&=\sigma_{\text{ess}}(\mathcal{C}_{\nu},H^{p}_{|\cdot|^{a}}(\mathbb{U}))\\
&=\left\lbrace z\in\mathbb{C}: \Bigg\vert z-\frac{p}{2(p\mbox{Re }\nu-a-1)}\Bigg\vert =\frac{p}{2(p\mbox{Re }\nu-a-1)}\right\rbrace .
\end{align*}
\noindent (ii) The Cesaro operator $\mathcal{C}_{\nu}$ is bounded on $A^{p}_{a}(\mathbb{U})$ for $p \cdot \mbox{Re } \nu>a+1$
with $\|C_{\nu}\|_{A^{p}_{a}}=p(\mbox{Re}\nu-a-1)^{-1}$ and 
\begin{align*}
\sigma(\mathcal{C}_{\nu},A^{p}_{a}(\mathbb{U}))&=\sigma_{\text{ess}}(\mathcal{C}_{\nu},A^{p}_{a}(\mathbb{U}))\\
&=\left\lbrace z\in\mathbb{C}: \Bigg\vert z-\frac{p}{2(p\mbox{Re }\nu-a-1)}\Bigg\vert =\frac{p}{2(p\mbox{Re }\nu-a-1)}\right\rbrace.
\end{align*}
\end{corollary}

\medskip
The spectral analysis of $\mathcal{H}_{\phi}$ on spaces of analytic functions have been already investigated, for example, by Abadias and Oliva-Maza \cite{ABADIAS2024110298, OlivaMaza2023}. In their seminal article, the authors described the spectrum of the infinitesimal generator of the strongly continuous group of compositions operators that defines $\mathcal{H}_\phi$. Subsequently, the spectral properties are transferred to $\mathcal{H}_{\phi}$ by employing the functional calculus of strip-type and sectorial-type operators. 
While their technique is undeniably elegant, it carries a disadvantage: applying the functional calculus requires specific conditions to be satisfied by the kernel function $\phi$.

The main original idea of this article is that, by applying an unitary operator $U\colon L^{p}(\mathbb{R},|\cdot|^{a}) \to L^{p}(\mathbb{R})$, the Hausdorff operator $H_\phi$ coincides with a convolution operator 
\[
(U \circ H_{\phi})f= (Uf)*k_{a,p}
\]
where
\[
k_{a,p}(t) = e^{\frac{a+1}{p}t}\phi(e^{t}), \quad t \in \mathbb{R}.
\]
This connection allows us to use classical literature results regarding the boundedness and the spectral properties of convolution operators. The situation becomes more complicate when we consider Hausdorff operators acting on spaces of holomorphic functions, since the description of $U(H_{|\cdot|^{a}}^{p}(\mathbb{U}))$ is not explicit. Nevertheless, by using the information obtained for the Lebesgue spaces, we still manage to characterize the spectrum of $\mathcal{H}_\phi$ in the holomorphic setting. 

\medskip
The remainder of this article is organized into six sections. In Section 1, we introduce the specific spaces of holomorphic functions on which the action of $\mathcal{H}_\phi$ is defined and we establish several abstract results required for the subsequent analysis. Section 2 is devoted to characterizing the spectrum of the Hausdorff operator on weighted Lebesgue spaces: specifically, we prove how this spectrum is linked to that of a convolution operator.
In Section 3, we provide the proof for Theorem \ref{Theorem spectrum weighted Hardy} and in Section 4 for Theorem \ref{Theorem spectrum Bergman}. We note that while these two results apply to different spaces, their proofs share a similar structure. In Section 5, we describe the spectral properties of the Cesàro-like operator.
Finally, in Section 6, we extend our study to the Hausdorff operator defined via a measure. Although this concluding section is not intended to be exhaustive, its purpose is to highlight a class of operators whose spectral properties are particularly compelling—especially when associated with measures with irregular spectra.

\medskip
%%%%%%%%%%%%%%%%%%%%%%%%%%%
\section{Preliminaries}

Inside the spectrum of the operator $A$, we distinguish specific subsets based on the behavior of $\lambda I - A$:

\begin{itemize}
    \item \emph{Point Spectrum $\sigma_p(A)$.} It consists of all the eigenvalues of $A$. 
    
    \item \emph{Essential Spectrum $\sigma_{\mathrm{ess}}(A)$.} It is defined as the set of points $\lambda$ where the operator $\lambda I-A$ fails to be Fredholm. Formally
    \[
    \sigma_{\mathrm{ess}}(A) = \{ \lambda \in \mathbb{C} \mid \dim(\ker(\lambda I - A)) = \infty \text{ or } \mathrm{codim}(\mathrm{Ran}(\lambda I - A)) = \infty \},
    \]
    where $\mathrm{codim}(Y) = \dim(X/Y)$ for a subspace $Y \subseteq X$.
    
    \item \emph{Approximate Point Spectrum $\sigma_{ap}(A)$.} It contains all $\lambda \in \mathbb{C}$ such that $\lambda I - A$ is not bounded from below. This occurs if there exists a sequence of unit vectors $\{x_n\} \subset X$ (where $\|x_n\| = 1$) such that
    \[
    \lim_{n \to \infty} \|(\lambda I - A)x_n\| = 0.
    \]
\end{itemize}
The resolvent set of $A$, denoted by $\rho(A)$, is the complement of the spectrum in the complex plane. It consists of all points $\lambda \in \mathbb{C}$ for which the resolvent operator $R(\lambda,A)=(\lambda I - A)^{-1}$ exists and is bounded.

\medskip
Our study focuses on the spectrum of the Hausdorff operator $\mathcal{H}_\phi$ acting on weighted Bergman and Hardy spaces of the upper half-plane $\mathbb{U}$.

Let $a > 0$ and $1 \leq p < \infty$. The {weighted Bergman space} $A^{p}_{a}(\mathbb{U})$ is the space of all the holomorphic functions $f$ on $\mathbb{U}$ such that
\[
\|f\|_{A^{p}_{a}} = \left( \int_{\mathbb{U}} |f(z)|^{p} \, dA_{a}(z) \right)^{1/p} < \infty,
\]
where $dA_{a}(z) = y^{a-1} \, dx dy$ for $z = x + iy$. The classical Hardy space $H^p(\mathbb{U})$ is formally interpreted as the limit case of weighted Bergman spaces as $a \to 0$, denoted by $H^p(\mathbb{U}) = A^{p}_{0}(\mathbb{U})$.

Following García-Cuerva \cite{JoseGarcia1979}, for $1 \leq p < \infty$ and $a > -1$, we define the {power weighted Hardy space} $H^{p}_{|\cdot|^{a}}(\mathbb{U})$ as the space of all the analytic functions on $\mathbb{U}$ satisfying
\[
\|f\|_{H^{p}_{|\cdot|^{a}}} = \sup_{y>0} \left( \int_{\mathbb{R}} |f(x + iy)|^p |x|^{a} \, dx \right)^{1/p} < \infty.
\]
For any $f \in H^{p}_{|\cdot|^{a}}(\mathbb{U})$, the non-tangential limit $f^*(x) = \lim_{y \to 0^+} f(x + iy)$ exists for almost every $x \in \mathbb{R}$. This boundary function $f^*$ satisfies
\[
\|f\|_{H^{p}_{|\cdot|^{a}}}^p = \int_{\mathbb{R}} |f^*(x)|^p |x|^{a} \, dx.
\]
According to \cite[Lemma II.1.2]{JoseGarcia1979}, for $f \in H^{p}_{|\cdot|^{a}}(\mathbb{U})$, we know the growth estimate
\begin{equation}\label{growth_Hardy}
|f(x + iy)| \leq \left[ w_{a}(B(x, y)) \right]^{-1/p} \|f\|_{H^{p}_{|\cdot|^{a}}},
\end{equation}
where, for any bounded measurable set $E \subset \mathbb{R}$, $w_{a}(E) = \int_{E} w_{a}(x) \, dx$, $w_{a}(x) = |x|^{a}$ and $B(x, t) = \{s \in \mathbb{R} : |s - x| < t\}$ for $x \in \mathbb{R}$ and $ t > 0$.  
Equation \eqref{growth_Hardy} implies that the convergence in the $H^{p}_{|\cdot|^{a}}$ norm yields the uniform convergence on compact subsets of $\mathbb{U}$.

\medskip
In the proof of our main theorems, we use the following abstract propositions.
\begin{prop}\label{subspace theorem R}
Let $X$ be a Banach space and $W$ be a closed subspace of $X$. Let $T:X\to X$ be bounded with $T(W)\subset W$. If for every $\lambda\in\rho(T)$, the resolvent $R(\lambda,T)$  satisfies $R(\lambda,T)(W)\subset W$, then 
$$
\sigma(T,W)\subseteq \sigma(T,X).
$$
\end{prop}
\begin{proof}
Let $\lambda \notin \sigma(T,X)$ and assume by contradiction that $\lambda \in \sigma(T,W)$.
There are two possibilities:
\begin{itemize}
    \item Either $(T-\lambda I)_{\vert_{W}}$ is not injective, that is, there exists $\psi \in W$ such that $T\psi-\lambda\psi=0$. However, since $W\subset X$, then $(T-\lambda I)_{\vert_{X}}$ is not injective as well, which contradicts the assumption $\lambda \notin \sigma(T,X)$.
    \item Or $(T-\lambda I)_{\vert_{W}}$ is not surjective, that is, there exists $\phi \in W$ non-zero such that $T\psi-\lambda\psi=\phi$ for no $\psi \in W$. However, since $\lambda \notin \sigma(T,X)$, there exists $\tilde{\psi} \in X$ such that $T\tilde{\psi}-\lambda\tilde{\psi}=\phi$. Since $\tilde{\psi}=R(\lambda,T)\phi \in W$, the operator $(T-\lambda I)_{\vert_{W}}$ needs to be surjective,
\end{itemize}
It follows that $\lambda \notin \sigma(T,W)$, that is,
$\sigma(T,W)\subseteq \sigma(T,V)$.
\end{proof}

\begin{prop}\label{subspace theorem}
Let $W$ be a closed subspace of  a Banach space $V$ such that $P: V \to W$ is the bounded projection. Let $T$ be a bounded operator on $V$ such that $T(W)\subseteq W$. If $T$ commutes with $P$, then 
$$
\sigma(T,W)\subseteq \sigma(T,V).
$$
\end{prop}
\begin{proof}
Let $\lambda \notin \sigma(T,V)$ and assume by contradiction that $\lambda \in \sigma(T,W)$.
There are two possibilities:
\begin{itemize}
    \item Either $(T-\lambda I)_{\vert_{W}}$ is not injective. In this case, we achieve a contradiction as done in Proposition \ref{subspace theorem R}.
    \item Or $(T-\lambda I)_{\vert_{W}}$ is not surjective, that is, there exists $\phi \in W$ non-zero such that $T\psi-\lambda\psi=\phi$ for no $\psi \in W$. However, since $\lambda \notin \sigma(T,V)$, there exists $\tilde{\psi} \in V$ such that $T\tilde{\psi}-\lambda\tilde{\psi}=\phi$.
    Consequently
    \[
    \phi=P(\phi)=P\circ T\tilde{\psi}-\lambda P\tilde{\psi}=(T-\lambda)P\tilde{\psi},
    \]
    which is impossible.
\end{itemize}
It follows that $\lambda \notin \sigma(T,W)$, that is,
$\sigma(T,W)\subseteq \sigma(T,V)$.
\end{proof}

\begin{comment}
    \noindent 
In Theorem \ref{Theorem spectrum weighted Hardy}, we will use the fact that $\mathcal{H}_\phi$ commutes with the Cauchy projection, the bounded projection from the Lebesgue space onto the Hardy space; In Theorem \ref{Theorem spectrum Bergman}, we will need the fact that $\mathcal{H}_\phi$ interchanges with the Bergman projection.
\medskip

\end{comment}
The following deep result is due to Wiener and usually established within
the framework of Banach algebras, see \cite[Pag. 104]{gelfand}. 

\begin{prop}\label{Theorem Wiener}
Let $K\in L^{1}(\mathbb{R})$ and suppose that $\lambda$ is a complex number such that $\lambda\not=0$, and $\lambda\not=\hat{K}(\xi)$ for any $\xi\in\mathbb{R}$. Then, there exists a function $A_\lambda \in L^{1} (\mathbb{R})$ such that
$$
\lambda A_{\lambda}(x)-K*A_{\lambda}(x)=K(x),\quad x\in\mathbb{R}
$$
where $K*A_{\lambda}(x)=\int_{\mathbb{R}}K(s)A_{\lambda}(x-s)ds$ is the usual convolution.
\end{prop}

\medskip
%%%%%%%%%%%%%%%%%%%%%%%%%%%

\section{Spectrum of $H_\phi$ on $L^{p}(\mathbb{R},|\cdot|^{a})$}

Before entering the proof of our main results, we analyze the spectrum of the Hausdorff operator $H_\phi$ acting on power weighted Lebesgue spaces of the real line. 

Let $1 \leq p \leq \infty$ and $a>-1$. The weighted Lebesgue space $L^{p}(\mathbb{R},|\cdot|^{a})$ consists by all the functions $f$ such that
\[
\|f\|_{L^{p}(\mathbb{R},|\cdot|^{a})}^{p}=\int_{\mathbb{R}} |f(x)|^p|x|^a dx<\infty, \quad 1\leq p<\infty,
\]
and, for $p=\infty$, by all $f$ such that $\|f\|_{L^{\infty}(\mathbb{R},|\cdot|^{a})}=\sup\{|f(x)|:x\in\mathbb{R}\}<\infty$.

The boundedness of the Hausdorff operator acting on weighted Lebesgue spaces was studied in \cite{Lizma14}, \cite{Miana21} and, afterwards, further developed in \cite{Hung2025}.

\begin{theorem}\cite[Theorem 1.5]{Hung2025}\label{Boundednes hausdorff in Lesbegue}
Let $1\leq p\leq\infty$ and $a>-1$. Assume that $\phi$ is a measurable function on $(0,\infty)$ such that 
\[
\int_{0}^{\infty}|\phi(t)| t^{\frac{1+a}{p}-1}dt.
\]
Then, $H_\phi$ is bounded on $L^{p}(\mathbb{R},|\cdot|^{a})$ and
\[
\left\vert \int_{0}^{\infty}\phi(t) t^{\frac{1+a}{p}-1}dt\right\vert \leq \| H_\phi\|_{L^{p}(\mathbb{R},|\cdot|^{a})} \leq \int_{0}^{\infty}|\phi(t)| t^{\frac{1+a}{p}-1}dt.
\]
\end{theorem} 
The key ingredient for the description of $\sigma(\mathcal{H}_\phi, L^{p}(\mathbb{R},|\cdot|^{a}) )$ is the fact that, after appropriate manipulation, the Hausdorff operator becomes a convolution operator and, for this reason, we are able to characterize its spectrum.

It is clear that 
\[
L^{p}(\mathbb{R},|\cdot|^a) \cong L^{p}(\mathbb{R}_{+},x^a\, dx) \oplus L^{p}(\mathbb{R}_{+}, x^a\,dx).
\]
Indeed, given a function $f \in L^{p}(\mathbb{R},|\cdot|^a)$, we consider the pair $(f_+, f_-)$ defined as 
\[
f_{+}(x) = f(x)\chi_{(0,\infty)}(x), \quad f_{-}(x) = f(-x)\chi_{(-\infty,0)}(-x).
\]
Clearly,
\[
\|f\|^{p}_{L^{p}(\mathbb{R},|\cdot|^a)} = \|f_{+}\|^{p}_{L^{p}(\mathbb{R}_{+},x^a\,dx)} + \|f_{-}\|^{p}_{L^{p}(\mathbb{R}_{+},x^a\,dx)}
\]
and the map $f \to (f_{+}, f_{-})$ provides an isometric isomorphism between $L^{p}(\mathbb{R},|\cdot|^a)$ and $L^{p}(\mathbb{R}_{+},x^a\, dx) \oplus L^{p}(\mathbb{R}_{+}, x^a\,dx)$. 

\begin{Lemm}\label{Lemma2}
The space $L^{p}(\mathbb{R}_{+},x^a\, dx)$ is isometrically isomorphic to $L^{p}(\mathbb{R})$.
\end{Lemm}

\begin{proof}
Let $f \in L^{p}(\mathbb{R}_{+},x^a\, dx)$ and define the map $U: L^{p}(\mathbb{R}_{+},x^a\, dx) \to L^{p}(\mathbb{R})$ as 
\[
Uf(x) = e^{\frac{a+1}{p}x}f(e^{x}).
\]
We observe that 
\begin{align*}
\|Uf\|^{p}_{L^{p}(\mathbb{R})} &= \int_{\mathbb{R}} |Uf(x)|^p \,dx = \int_{\mathbb{R}} |f(e^{x})|^{p} e^{(a+1)x} \,dx \\
&= \int_{0}^{\infty} |f(u)|^{p} u^a \,du = \|f\|^{p}_{L^{p}(\mathbb{R}_{+},x^a\,dx)}.
\end{align*}
This shows that $U$ is an isometry. The operator $U$ is also surjective. Indeed, for $g \in L^{p}(\mathbb{R})$, set $f(x) = x^{-\frac{a+1}{p}} g(\ln x)$ for $x \in \mathbb{R}_{+}$. Then $Uf = g$. The inverse operator $U^{-1}: L^{p}(\mathbb{R}) \to L^{p}(\mathbb{R}_{+},x^a\, dx)$ is defined as $U^{-1}g(x) = x^{-\frac{a+1}{p}} g(\ln x)$. The case $p=\infty$ follows by  obvious modifications.
\end{proof}

The Hausdorff operator respects the decomposition $L^{p}(\mathbb{R}_{+},x^a \, dx)  \oplus  L^{p}(\mathbb{R}_{+}, x^a \,dx)$. Indeed, $H_{\phi}f_{+}, H_{\phi}f_{-}$ are supported on $(0,\infty)$, and we may write $H_{\phi} = H_{\phi}^{+} \oplus H_{\phi}^{-}$.
For this reason, it is sufficient to consider the action of the Hausdorff operator $H_{\phi}$ on $L^{p}(\mathbb{R}_{+},x^a\, dx)$.

\begin{Lemm}\label{Lemma 3}
Let $H_\phi$ be a Hausdorff operator on $L^{p}(\mathbb{R}_{+},x^a\,dx)$ and let $U$ be defined as in Lemma \ref{Lemma2}. Then, for every $f \in L^{p}(\mathbb{R})$,
\[
(UH_\phi U^{-1})f(x) = \int_{\mathbb{R}} f(x-s)k_{a,p}(s) \, ds,
\]
where $k_{a,p}(s) = \phi(e^{s})e^{\frac{a+1}{p}s}$.
\end{Lemm}

\begin{proof}
Let $f \in L^{p}(\mathbb{R})$ and set $g = U^{-1}f$. Then 
\begin{align*}
(UH_{\phi} g)(x) &= e^{\frac{a+1}{p}x} \int_{0}^{\infty} g\left(\frac{e^{x}}{t}\right) \frac{\phi(t)}{t} \, dt \\
&= \int_{0}^{\infty} e^{\frac{a+1}{p}(x-\ln t)} g\left(\frac{e^{x}}{t}\right) t^{\frac{a+1}{p}-1} \phi(t) \, dt \\
&= \int_{0}^{\infty} Ug(x-\ln t) \, t^{\frac{a+1}{p}-1} \phi(t) \, dt.
\end{align*}
Let $s = \ln t$, so $ds = t^{-1} dt$. Then 
\[
(UH_{\phi} U^{-1})f(x) = \int_{\mathbb{R}} f(x-s)k_{a,p}(s) \, ds.
\]
This completes the proof.
\end{proof}
We are now ready to describe the spectrum of ${H}_\phi$ acting on $L^{p}(\mathbb{R},|\cdot|^a)$. We recall that, whenever it is well defined, the Fourier transform of $k$ is given by  $\hat{k}(x)=\int_{\mathbb{R}}k(t)e^{-itx}dt$.

\begin{theorem}\label{Spectrum L^p}
Let $1 \leq p \leq \infty$ and $a > -1$. If $\int_{0}^{\infty}|\phi(t)| t^{\frac{1+a}{p}-1}dt$ is finite,
then 
\[
\sigma(H_\phi, L^{p}(\mathbb{R},|\cdot|^a)) = \overline{\widehat{k_{a,p}}(\mathbb{R})},
\]
where $k_{a,p}(t) = e^{\frac{a+1}{p}t}\phi(e^{t})$.
\end{theorem}

\begin{proof}
Since $L^{p}(\mathbb{R},|\cdot|^a) \cong L^{p}(\mathbb{R}_{+},x^a\, dx) \oplus L^{p}(\mathbb{R}_{+}, x^a\,dx)$, we have that
\[
\sigma(H_{\phi}, L^{p}(\mathbb{R}, |\cdot|^{a})) = \sigma(H_{\phi}, L^{p}(\mathbb{R}_+, x^{a}\,dx)).
\]
According to Theorem \ref{Boundednes hausdorff in Lesbegue}, the Hausdorff operator $H_\phi$ is bounded on  $L^{p}(\mathbb{R}, |\cdot|^{a})$ and the function $k_{a,p} \in L^1(\mathbb{R})$. 
For $f \in L^{p}(\mathbb{R})$, let $$
Kf(x) = \int_{\mathbb{R}} f(x-s)k_{a,p}(s) \, ds
$$ be a convolution operator induced by $k_{a,p}$. It is well known \cite{Boyd_1973} that the spectrum of $K$ satisfies  $\sigma(K, L^{p}(\mathbb{R})) = \overline{\widehat{k_{a,p}}(\mathbb{R})}$. By the unitary equivalence established in Lemma \ref{Lemma 3}, we conclude that
\[
\sigma(H_{\phi}, L^{p}(\mathbb{R}, |\cdot|^{a})) = \sigma(K, L^{p}(\mathbb{R})) = \overline{\widehat{k_{a,p}}(\mathbb{R})}.
\]
\end{proof}

Before starting the analysis of the spectrum of the Hausdorff operators acting on spaces of analytic functions, we observe that the description of the Hausdorff operator through convolution operator as done in Lemma \ref{Lemma 3}, may be also helpful in describing the norm of $H_\phi$ on $L^{p}(\mathbb{R},|\cdot|^a)$. 

\begin{prop}\label{norm hausdorf L^2}
Let $\phi$ be measurable on $[0,\infty)$. The Hausdorff operator $H_{\phi}: L^{2}(\mathbb{R},|\cdot|^a) \to L^{2}(\mathbb{R},|\cdot|^a)$ is bounded if and only if 
\begin{equation}
\widehat{k_{a,2}}(\xi) = \int_{\mathbb{R}} k_{a,2}(s)e^{-i\xi s} \, ds = \int_{0}^{\infty} t^{\frac{a+1}{2}} \phi(t) \frac{dt}{t^{1+i\xi}} \in L^\infty(\mathbb{R}).
\end{equation}
Moreover, $\|H_{\phi}\|_{L^{2}(\mathbb{R},|\cdot|^a)} = \|\widehat{k_{a,2}}\|_{L^{\infty}(\mathbb{R})}$.
\end{prop}

\begin{proof}
If $H_\phi$ is bounded on $L^{2}(\mathbb{R},|\cdot|^a)$, then it is bounded on $L^{2}(\mathbb{R}_{+},x^a\, dx)$. Therefore, $U{H}_\phi U^{-1}$ is a bounded convolution operator on $L^{2}(\mathbb{R})$. This fact and the Plancherel theorem imply $\widehat{k_{a,2}} \in L^\infty(\mathbb{R})$ and $\|\widehat{k_{a,2}}\|_{L^{\infty}(\mathbb{R})} \leq \|H_{\phi}\|$. Conversely, if $\widehat{k_{a,2}} \in L^\infty(\mathbb{R})$, then $UH_\phi U^{-1}$ is bounded on $L^{2}(\mathbb{R})$, and the norm identity follows from the decomposition of $L^2(\mathbb{R}, |\cdot|^a)$.
\end{proof}

\medskip
%%%%%%%%%%%%%%%%%%%%%%%%%%%
\section{Proof of Theorem \ref{Theorem spectrum weighted Hardy}}

Let $\phi$ be a measurable function on $(0, \infty)$. We consider  bounded Hausdorff operator $\mathcal{H}_\phi$ on the weighted Hardy spaces of the upper half plane. Hung and Dy \cite{Hung2025} characterized its boundedness as follows.

\begin{theorem}[{\cite[Theorem. 1.3 ]{Hung2025}}]\label{Boundedness condition Hardy}
Let $1\leq p\leq\infty$, $a > -1$ and let $\phi$ be a measurable function on $(0, \infty)$.
\begin{enumerate}
    \item[$(i)$] If $\phi$ satisfies the condition
    \begin{equation}\label{Hardy BC}
    \int_{0}^{\infty} t^{\frac{a+1}{p}-1} |\phi(t)| \, dt < \infty,
    \end{equation}
    then $\mathcal{H}_{\phi}$ is bounded on $H^{p}_{|\cdot|^{a}}(\mathbb{U})$. Moreover,
    \[
    \left| \int_{0}^{\infty} t^{\frac{a+1}{p}-1} \phi(t) \, dt \right| \leq \|\mathcal{H}_{\phi}\|^p_{H^{p}_{|\cdot|^{a}}} \leq \int_{0}^{\infty} t^{\frac{a+1}{p}-1} |\phi(t)| \, dt.
    \]
    \item[$(ii)$] Conversely, if $\phi \geq 0$ and $\mathcal{H}_{\phi}$ is bounded on $H^{p}_{|\cdot|^{a}}(\mathbb{U})$, then condition \eqref{Hardy BC} holds.
\end{enumerate}
\end{theorem}

We are ready to describe the spectrum of $\mathcal{H}_\phi$ on $H^p_{|\cdot|^a}(\mathbb{U})$.
\begin{proof}[Proof of Theorem \ref{Theorem spectrum weighted Hardy}]
For the sake of shortness we write $\hat{k}$ instead of $\widehat{k_{a,p}}$. 
Let $\lambda\not \in \overline{\hat{k}(\mathbb{R})}$. According to Proposition \ref{Theorem Wiener} and \cite[Lemma 1]{Boyd_1973}, there exists a function $A_{\lambda} \in L^1 (\mathbb{R})$ such that
$\lambda^{-1}(I+A_{\lambda}*)$ is the inverse of $(\lambda I -k*)$ on $L^{p}(\mathbb{R})$. Therefore, $\lambda^{-1}(I+U^{-1}A_{\lambda}*U)$ is the inverse of $(\lambda I -U^{-1}k*U)$ on $L^{p}(\mathbb{R}_{+}, x^a \,dx)$. But

$$
U^{-1}k*U=H_\phi\quad\mbox{and}\quad U^{-1}A_{\lambda}*U=H_{\psi_\lambda}
$$
with $\psi_{\lambda}(s)=A_{\lambda}(\ln s)s^{-\frac{a+1}{p}}$.  Observe that
\begin{align*}
\int_{0}^{\infty}|\psi_{\lambda}(s)|s^{\frac{a+1}{p}-1}\,ds&=\int_{0}^{\infty}|A_{\lambda}(\ln s)|\frac{1}{s}\,ds=\int_{\mathbb{R}}|A_{\lambda}(s)|\,ds<\infty.
\end{align*}
This implies that  $H_{\psi_\lambda}$ is bounded on $L^{p}(\mathbb{R}, |x|^a\,dx)$ 
%(and on $L^{p}(\mathbb{R}_{+}, x^a\,dx))$
, see Theorem \ref{Boundednes hausdorff in Lesbegue}, and that $\mathcal{H}_{\psi_\lambda}$ is bounded  on $H^{p}_{|\cdot|^{a}}(\mathbb{U})$, see Theorem \ref{Boundedness condition Hardy}. Moreover by \cite[Theorem 1.4]{Hung2025}  
$$
(\mathcal{H}_{\phi}f)^*=H_{\phi}f^*\quad \mbox{and}\quad  (\mathcal{H}_{\psi}f)^*=H_{\psi}f^*,
$$ for every $f\in H^{p}_{|\cdot|^{a}}(\mathbb{U})$. Thus 
\begin{align*}
 (A)&\quad (\lambda I-\mathcal{H}_\phi )(H^{p}_{|\cdot|^{a}}(\mathbb{U}))\subseteq H^{p}_{|\cdot|^{a}}(\mathbb{U}), \\
 (B)&\quad \frac{1}{\lambda}( I+\mathcal{H}_\psi) (H^{p}_{|\cdot|^{a}}(\mathbb{U}))\subseteq H^{p}_{|\cdot|^{a}}(\mathbb{U}). 
\end{align*}
Moreover, for every $f\in H^{p}_{|\cdot|^{a}}(\mathbb{U}) $
\begin{align*}
 &\|   \frac{1}{\lambda}(I+\mathcal{H}_{\psi_{\lambda}})(\lambda I-\mathcal{H}_{\phi})f-f\|_{H^{p}_{|\cdot|^{a}}(\mathbb{U})}\\
 &\quad \leq C\|  ( \frac{1}{\lambda}(I+\mathcal{H}_{\psi_{\lambda}})(\lambda I-\mathcal{H}_{\phi})f)^{*}-f^{*}\|_{L^{p}(\mathbb{R},|\cdot|^a dx)}\\
 &\quad =C\|  ( \frac{1}{\lambda}(I+H_{\psi_{\lambda}})(\lambda I-H_{\phi})f^*-f^{*}\|_{L^{p}(\mathbb{R},|\cdot|^a dx)}.
\end{align*}
 We consider the pair $(f_+^{*}, f_-^{*})$ defined as 
\[
f_{+}^{*}(x) = f^{*}(x)\chi_{(0,\infty)}(x), \quad f_{-}^{*}(x) = f^{*}(-x)\chi_{(-\infty,0)}(-x).
\]
Clearly,
\begin{align*}
    &\|( \frac{1}{\lambda}(I+H_{\psi_{\lambda}})(\lambda I-H_{\phi})f^*-f^{*}\|^{p}_{L^{p}(\mathbb{R},|\cdot|^a)}\\
    &\quad = \|( \frac{1}{\lambda}(I+H_{\psi_{\lambda}})(\lambda I-H_{\phi})f_{+}^{*}-f_{+}^{*}\|^{p}_{L^{p}(\mathbb{R}_{+},x^a\,dx)}\\
    & \qquad + \|( \frac{1}{\lambda}(I+H_{\psi_{\lambda}})(\lambda I-H_{\phi})f_{-}^{*}-f_{-}^{*}\|^{p}_{L^{p}(\mathbb{R}_{+},x^a\,dx)}=0.
\end{align*}

Similarly $\|   \frac{1}{\lambda}(\lambda I-\mathcal{H}_{\phi})(I+\mathcal{H}_{\psi_{\lambda}})f-f\|_{H^{p}_{|\cdot|^{a}}(\mathbb{U})}=0$. This implies that
 $$
 \frac{1}{\lambda}(I+\mathcal{H}_{\psi_{\lambda}})(\lambda I-\mathcal{H}_{\phi})=(\lambda I-\mathcal{H}_\phi)\frac{1}{\lambda}(I+\mathcal{H}_{\psi_{\lambda}})=I, \mbox{on } H^{p}_{|\cdot|^{a}}(\mathbb{U})).
 $$
 Consequently, $\lambda$ belongs to the resolvent set of ${\mathcal{H}_{\phi}}{\vert_{H^{p}_{|\cdot|^{a}}}}$ and, by Proposition \ref{subspace theorem R},
$$
\sigma(\mathcal{H}_{\phi}, H^{p}_{|\cdot|^{a}}(\mathbb{U}))\subset\overline{\hat{k}(\mathbb{R})}.
$$

For the reverse inclusion, let $\delta \in (0,1)$ and let $\phi_{\delta}(t) = \phi(t)\chi_{[\delta,1/\delta)}(t)$. Then 
$$
\|\mathcal{H}_{\phi}-\mathcal{H}_{\phi_{\delta}}\|_{H^p_{|\cdot|^a}}\leq \int_{0}^{\delta}|\phi(t)|t^{\frac{a+1}{p}-1}\,dt+\int_{\frac{1}{\delta}}^{\infty}|\phi(t)|t^{\frac{a+1}{p}-1}\,dt
$$
which goes to zero as $\delta$ tends to zero.
For any $0<\epsilon<1$ and $\xi\in\mathbb{R}$, we define the function
$$
f_{\epsilon,\xi}(z)=(z+i)^{-\frac{a+1}{p}-\epsilon+i\xi}.
$$
Since $e^{-|\xi|\pi}\leq|(z+i)^{i\xi}|\leq e^{|\xi|\pi}$, by \cite[Pag.17]{Hung2025} we have that
$$
\|f_{\epsilon,\xi}\|_{H^{p}_{|\cdot|^{a}}}\sim\|f_{\epsilon,0}\|_{H^{p}_{|\cdot|^{a}}}\sim \epsilon^{-1/p}<\infty,
$$
where the constants involved in the first comparison depend on $\xi$.
Let us consider
\[
\widehat{k_\delta}(\xi)=\int_{0}^{\infty}\phi_{\delta}(t)t^{\frac{a+1}{p}}\frac{dt}{t^{1+i\xi}}.
\]
For every $z = x + iy\in\mathbb{U}$, we have that 
\begin{align*}
\mathcal{H}_{\phi_{\delta}}f_{\epsilon,\xi}(z)-f_{\epsilon,\xi}(z)\widehat{k_\delta}(\xi)=&\mathcal{H}_{\phi_{\delta}}f_{\epsilon,\xi}(z)-f_{\epsilon,\xi}(z)\int_{0}^{\infty}\phi_{\delta}(t)t^{\frac{a+1}{p}}\frac{dt}{t^{1+i\xi}}\\
=&\int_{\delta}^{1/\delta}\left(\varphi_{\epsilon,\xi,z}(t)-\varphi_{\epsilon,\xi,z}(1)\right)\phi(t)t^{\frac{a+1}{p}}\frac{dt}{t^{1+i\xi}},
\end{align*}
where $\varphi_{\epsilon,\xi,z}(t)=t^{\epsilon}/ {(z+ti)^{\frac{1+a}{p}+\epsilon-i\xi}}$.
For every $t \in (\delta, 1/\delta)$ and $z \in \mathbb{U}$, the Lagrange mean value theorem implies that
\begin{align*}
   |\varphi_{\epsilon,\xi,z}(t)-\varphi_{\epsilon,\xi,z}(1)|\leq |t-1|\sup_{s\in (\delta,1/\delta)} |\varphi_{\epsilon,\xi,z}'(s)|.
\end{align*}
We calculate
\begin{align*}
    \varphi_{\epsilon,\xi,z}'(t)&=\frac{\epsilon t^{\epsilon-1}}{(z+ti)^{\frac{1+a}{p}+\epsilon-i\xi}}-i\frac{(\frac{1+a}{p}+\epsilon-i\xi)t^{\epsilon}}{(z+ti)^{\frac{1+a}{p}+1+\epsilon-i\xi}}\\
    &=\left(\frac{\epsilon t^{\epsilon-1}}{(z+ti)^{\frac{1+a}{p}+\epsilon}}-i\frac{(\frac{1+a}{p}+\epsilon-i\xi)t^{\epsilon}}{(z+ti)^{\frac{1+a}{p}+1+\epsilon}}\right)(z+it)^{i\xi}.
\end{align*}
Since $|(z+it)^{i\xi}|=e^{-\xi \arg{(z+it)}}\leq e^{|\xi|\pi}$, we have that
\begin{align*}
   |\varphi_{\epsilon,\xi,z}(t)-\varphi_{\epsilon,\xi,z}(1)|\leq (\frac{1}{\delta}-1)e^{|\xi|\pi}\left(\frac{\epsilon\delta^{\epsilon-1}}{|z+\delta i|^{\frac{a+1}{p}+\epsilon}}+\frac{(\frac{1+a}{p}+\epsilon+|\xi|)\delta^{-\epsilon}}{|z+\delta i|^{1+\frac{a+1}{p}+\epsilon}}\right).
\end{align*}
and 
\begin{align*}
   |\varphi_{\epsilon,\xi,z}(t)-\varphi_{\epsilon,\xi,z}(1)|\leq e^{|\xi|\pi}\left(\frac{\epsilon\delta^{-(2+\frac{a+1}{p})}}{|z+ i|^{\frac{a+1}{p}+\epsilon}}+\frac{(\frac{1+a}{p}+\epsilon+|\xi|)\delta^{-(2+2\epsilon+\frac{a+1}{p})}}{|z+ i|^{1+\frac{a+1}{p}+\epsilon}}\right),
\end{align*}
since $|z+\delta i|>\delta|z+i|$ for any $z\in\mathbb{U}$.
Consequently
\begin{align*}
&\frac{\|\mathcal{H}_{\phi_{\delta}}(f_{\epsilon,\xi})- \widehat{k_{\delta}}(\xi)f_{\epsilon,\xi} \|_{H^{p}_{|\cdot|^{a}}} }{\|f_{\epsilon,\xi}\|_{H^{p}_{|\cdot|^{a}}} }\\
&\qquad \qquad \leq C\int_{\delta}^{1/\delta}|\phi(t)| t^{\frac{a+1}{p}-1}\,dt\left[\epsilon\delta^{\delta^{-(2+\frac{a+1}{p})}}+\epsilon^{1/p}(\frac{1+a}{p}+\epsilon+|\xi|)\right]
\end{align*}
which goes to zero when $\epsilon$ tends to $0$.
More generally, let $\varepsilon>0$ small and pick $\delta>0$ such that $\|\mathcal{H}_{\phi}-\mathcal{H}_{\phi_{\delta}}\|_{H^{p}_{|\cdot|^{a}} }<\varepsilon$ and $|\hat{k}(\xi)-\widehat{k_\delta}(\xi)|<\varepsilon$. Then 
\begin{align*}
&\lim_{\epsilon\to 0}\frac{\|\mathcal{H}_{\phi}f_{\epsilon,\xi}- \hat{k}(\xi)f_{\epsilon,\xi} \|_{H^{p}_{|\cdot|^{a}}}} {\|f_{\epsilon,\xi}\|_{H^{p}_{|\cdot|^{a}} }}\\
&\quad \leq \lim_{\epsilon\to0}\left(\frac{\|\mathcal{H}_{\phi_{\delta}}(f_{\epsilon,\xi})- \widehat{k_\delta}(\xi)f_{\epsilon,\xi}\|_{H^{p}_{|\cdot|^{a}} }}{\|f_{\epsilon,\xi}\|_{H^{p}_{|\cdot|^{a}} }}+\|\mathcal{H}_{\phi}-\mathcal{H}_{\phi_{\delta}}\|_{H^{p}_{|\cdot|^{a}} }+ |\hat{k}(\xi)-\widehat{k_\delta}(\xi)|\right)\\
&\quad<2\varepsilon.
\end{align*}
Since $\varepsilon$ is arbitrary, this implies that $\hat{k}(\xi)$ is in the approximate point spectrum of $\mathcal{H}_{\phi}$, i.e 
$$
\sigma(\mathcal{H}_{\phi}, H^{p}_{|\cdot|^{a}}(\mathbb{U}))=\overline{\hat{k}(\mathbb{R})}.
$$
\begin{comment}
    
Since $k\in L^{1}(\mathbb{R})$, $\hat{k}(\xi)$ is continuous and vanishes at infinity. This implies that every point of  $\sigma(\mathcal{H}_{\phi}, H^{p}_{|\cdot|^{a}}(\mathbb{U}))$ is an accumulation point of $\sigma(\mathcal{H}_{\phi}, H^{p}_{|\cdot|^{a}}(\mathbb{U}))$ and that the range  of $\hat{k}$ is a curve with empty interior. Therefore every point of  $\sigma(\mathcal{H}_{\phi}, H^{p}_{|\cdot|^{a}}(\mathbb{U}))$ is an accumulation point of $\rho(\mathcal{H}_{\phi}, H^{p}_{|\cdot|^{a}}(\mathbb{U}))$. Thus 
by \cite[Remark 4.1]{OlivaMaza2023}
$$
\sigma(\mathcal{H}_{\phi}, H^{p}_{|\cdot|^{a}}(\mathbb{U}))=\sigma_{ess}(\mathcal{H}_{\phi}, H^{p}_{|\cdot|^{a}}(\mathbb{U})).
$$
\end{comment}

\end{proof}

\begin{comment}

For $-1<a<p-1$, the Hausdorff operator $\mathcal{H}_\phi$ commutes with the Cauchy projection on $L^p(\mathbb{R},|\cdot|^a)$, see \cite[Theorem 1.5]{Hung2025}. For this reason and due to the fact that $\sigma(\mathcal{H}_\phi, L^p(\mathbb{R},|\cdot|^a))=\overline{\widehat{k_{a,p}}(\mathbb{R})}$, one can easily prove that $ \sigma(\mathcal{H}_\phi, H^p_{|\cdot|^a})\subseteq\overline{\widehat{k_{a,p}}(\mathbb{R})}$ through Proposition \ref{subspace theorem}.
\end{comment}

The second part of the proof of Theorem \ref{Theorem spectrum weighted Hardy} provides an original lower bound for the norm of the Hausdorff operator.

\begin{corollary}\label{bound from below}
Let $\phi$ be a measurable function in $[0,\infty)$ which satisfies
$$
\int_{0}^{\infty} |\phi(t)|t^{\frac{a+1}{p}-1}\,dt<\infty.
$$
Then 
$$
\sup_{\xi \in {\mathbb{R}}} \left\vert \widehat{k_{a,p}}(\xi) \right\vert \leq\|\mathcal{H}_{\phi}\|_{H^{p}_{|\cdot|^{a}}},
$$
where $\widehat{k_{a,p}}(\xi)=\int^{\infty}_{0}\phi(t)t^{\frac{a+1}{p}}\frac{dt}{t^{1+i\xi}}$. 
\end{corollary}
\begin{proof}
It is true that
\begin{align*}
 | \widehat{k_{a,p}}(\xi)|&=\frac{\| \widehat{k_{a,p}}(\xi)f_{\epsilon,\xi}\|_{H^p_{|\cdot|^a}}}{\|f_{\epsilon,\xi}\|_{H^p_{|\cdot|^a}}}\leq   \frac{\| \mathcal{H}_{\phi}f_{\epsilon,\xi}\|_{H^p_{|\cdot|^a}}} {\|f_{\epsilon,\xi}\|_{H^p_{|\cdot|^a}}}+ \frac{\|\mathcal{H}_{\phi}f_{\epsilon,\xi}- \hat{k}(\xi)f_{\epsilon,\xi} \|_{H^{p}_{|\cdot|^{a}}}} {\|f_{\epsilon,\xi}\|_{H^{p}_{|\cdot|^{a}} }}\\
 &\leq \|\mathcal{H}_\phi\|_{H^{p}_{|\cdot|^{a}} }+\frac{\|\mathcal{H}_{\phi}f_{\epsilon,\xi}- \hat{k}(\xi)f_{\epsilon,\xi} \|_{H^{p}_{|\cdot|^{a}}}} {\|f_{\epsilon,\xi}\|_{H^{p}_{|\cdot|^{a}} }}.
\end{align*}
The second term, by the proof of the previous Theorem, goes to zero as $\epsilon$ vanishes, for every fixed $\xi$. This proves that   $| \widehat{k_{a,p}}(\xi)|\leq \|\mathcal{H}_\phi\|_{H^{p}_{|\cdot|^{a}} }$.
 \end{proof}

 As an immediate corollary of Proposition \ref{norm hausdorf L^2} and Corollary \ref{bound from below}, we have that $\|\mathcal{H}_\phi\|_{H^2_{|\cdot|^a}}=\sup_{\xi \in {\mathbb{R}}} \left\vert \widehat{k_{a,2}}(\xi) \right\vert $. The lower estimate provided in Corollary \ref{bound from below} may be higher than the lower bound provided in Theorem \ref{Boundedness condition Hardy}.

\medskip
%%%%%%%%%%%%%%%%%%%%%%%%%%%

\section{Proof of Theorem \ref{Theorem spectrum Bergman}}
Let $\phi$ be a measurable function on $(0, \infty)$. We consider  bounded Hausdorff operator $\mathcal{H}_\phi$ on the weighted Bergman spaces of the upper half plane. Hung and Dy \cite{Hung2025} characterized its boundedness as follows.

\begin{theorem}[{\cite[Theorem. 1.1 ]{Hung2025}}]\label{boundedness condition Bergman}
Let $1\leq p<\infty$ and $a > 0$ and let $\phi$ be a measurable function on $(0, \infty)$.
\begin{enumerate}
    \item[$(i)$] If $\phi$ satisfies the condition
    \begin{equation}\label{Boundedness condition Bergman}
    \int_{0}^{\infty} t^{\frac{a+1}{p}-1} |\phi(t)| \, dt < \infty,
    \end{equation}
    then $\mathcal{H}_{\phi}$ is bounded on $A^{p}_{a}(\mathbb{U})$. Moreover,
    \[
    \left| \int_{0}^{\infty} t^{\frac{a+1}{p}-1} \phi(t) \, dt \right| \leq \|\mathcal{H}_{\phi}\|^p_{A^{p}_{a}} \leq \int_{0}^{\infty} t^{\frac{a+1}{p}-1} |\phi(t)| \, dt.
    \]
    \item[$(ii)$] Conversely, if $\phi \geq 0$ and $\mathcal{H}_{\phi}$ is bounded on $A^{p}_{a}(\mathbb{U})$, then condition \eqref{Boundedness condition Bergman} holds.
\end{enumerate}
\end{theorem}

We are ready to describe the spectrum of $\mathcal{H}_\phi$ on $A^p_a$.

\begin{proof}[Proof of Theorem \ref{Theorem spectrum Bergman}]

If $f\in L^{p}_{a}(\mathbb{U})$, then 
\begin{align*}
\|f\|_{L^{p}_{a}(\mathbb{U})}^{p} 
&=\int_{0}^{\pi} (\sin\theta)^{a}\|f_{\theta}\|^{p}_{L^{p}(\mathbb{R}_{+}r^{a}\,dr)}d\theta,
\end{align*}
where $f_{\theta}(r)=f(re^{i\theta})$.
Thus $f\in L^{p}_{a}(\mathbb{U})$ implies that for almost every $\theta\in[0,\pi]$, the function $f_\theta\in L^{p}(\mathbb{R}_{+},r^{a}\,dr)$.
Observe that $(\mathcal{H}_{\phi}f)_{\theta}=\mathcal{H}_{\phi}f_{\theta}$.

For the sake of shortness we write $\hat{k}$ instead of $\widehat{k_{a,p}}$.
Let $\lambda\not \in \overline{\hat{k}(\mathbb{R})} $. According to Proposition \ref{Theorem Wiener},  there exists a function $A_{\lambda}, \in L^1 (\mathbb{R})$ such that
$\lambda^{-1}(I+A_{\lambda}*)$ is the inverse of $(\lambda I -k*)$ on $L^{p}(\mathbb{R})$. Therefore the operator $\lambda^{-1}(I+U^{-1}A_{\lambda}*U)$ is the inverse of $(\lambda I -U^{-1}k*U)$ on $L^{p}(\mathbb{R}_{+}, r^{a} \,dr)$. Since $U^{-1}k*U=\mathcal{H}_\phi$, $U^{-1}A_{\lambda}*U=\mathcal{H}_{\psi_\lambda}$ with $\psi_{\lambda}(s)=A_{\lambda}(\ln s)s^{-\frac{a+1}{p}}$, and 
\begin{align*}
\int_{0}^{\infty}|\psi_{\lambda}(s)|s^{\frac{a+1}{p}-1}\,ds&=\int_{0}^{\infty}|A_{\lambda}(\ln s)|\frac{1}{s}\,ds\\
&=\int_{\mathbb{R}}|A_{\lambda}(s)|\,ds<\infty,
\end{align*}
we have that  $\mathcal{H}_{\psi_\lambda}$ is bounded on spaces $L^{p}(\mathbb{R}_{+}, r^{a}\,dr)$, see Theorem \ref{Boundednes hausdorff in Lesbegue}, on $L^{p}_{a}(\mathbb{U})$ (due to a simple application of Minkowsky inequality),  and on $A^{p}_{a}(\mathbb{U})$, see Theorem \ref{boundedness condition Bergman}.
Thus 
\begin{align*}
 (A)&\quad (\lambda I-\mathcal{H}_\phi )(L^{p}_{a}(\mathbb{U}))\subseteq L^{p}_{a}(\mathbb{U}), \\
 (B)&\quad \frac{1}{\lambda}( I+\mathcal{H}_\psi)( A^{p}_{a}(\mathbb{U}))\subseteq A^{p}_{a}(\mathbb{U}). 
\end{align*}

Moreover $\frac{1}{\lambda}(I+\mathcal{H}_{\psi_{\lambda}})(\lambda I-\mathcal{H}_{\phi})f_{\theta}=(\lambda I-\mathcal{H}_\phi)\frac{1}{\lambda}(I+\mathcal{H}_{\psi_{\lambda}})f_{\theta}=f_{\theta}$ for each $\theta\in[0,\pi]$. This implies that

\begin{align*}
\|  \frac{1}{\lambda}(I&+\mathcal{H}_{\psi_{\lambda}})(\lambda I-\mathcal{H}_{\phi})f-f\|_{L^{p}_{a}(\mathbb{U})}\\
&= \int_{0}^{\pi} (\sin\theta)^{a}\|(\frac{1}{\lambda}(I+\mathcal{H}_{\psi_{\lambda}})(\lambda I-\mathcal{H}_{\phi})f-f)_{\theta}\|^{p}_{L^{p}(\mathbb{R}_{+}r^{a}\,dr)}d\theta\\
&=\int_{0}^{\pi} (\sin\theta)^{a}\|(\frac{1}{\lambda}(I+\mathcal{H}_{\psi_{\lambda}})(\lambda I-\mathcal{H}_{\phi})f)_{\theta}-f_{\theta}\|^{p}_{L^{p}(\mathbb{R}_{+}r^{a}\,dr)}d\theta\\
&=\int_{0}^{\pi} (\sin\theta)^{a}\|\frac{1}{\lambda}(I+\mathcal{H}_{\psi_{\lambda}})(\lambda I-\mathcal{H}_{\phi})f_{\theta}-f_{\theta}\|^{p}_{L^{p}(\mathbb{R}_{+}r^{a}\,dr)}d\theta\\
&=0,
\end{align*}
and, with similar computations, that 
$$
\|  \frac{1}{\lambda}(I+\mathcal{H}_{\phi})(\lambda I-\mathcal{H}_{\psi_{\lambda}})f-f\|_{L^{p}_{a}(\mathbb{U})}=0.
$$
Since
$$
 \frac{1}{\lambda}(I+\mathcal{H}_{\phi})(\lambda I-\mathcal{H}_{\psi_{\lambda}})= \frac{1}{\lambda}(I+\mathcal{H}_{\psi_{\lambda}})(\lambda I-\mathcal{H}_{\phi})=I
$$
on $L^{p}_{a}(\mathbb{U})$, by the inclusion relations $(A)$ and $(B)$, the point $\lambda$ belongs to the resolvent sets of $\rho(\mathcal{H}_\phi,L^{p}_{a}(\mathbb{U}))$  and   $\rho(\mathcal{H}_\phi,A^{p}_{a}(\mathbb{U}))$. Thus
$$
\sigma(\mathcal{H}_{\phi}, L^{p}_{a}(\mathbb{U}))\subset \overline{\hat{k}(\mathbb{R})}\quad \mbox{and}\quad \sigma(\mathcal{H}_{\phi}, A^{p}_{a}(\mathbb{U}))\subset \overline{\hat{k}(\mathbb{R})}.
$$
For the reverse inclusion, let $\delta \in (0,1)$ and let $\phi_{\delta}(t) = \phi(t)\chi_{[\delta,1/\delta)}(t)$. Then 
$
\|\mathcal{H}_{\phi}-\mathcal{H}_{\phi_{\delta}}\|_{A^p_\alpha}
$
goes to zero as $\delta$ tends to zero.
For any $0<\epsilon<1$ and $\xi\in\mathbb{R}$, we define the function
$$
f_{\epsilon,\xi}(z)=(z+ i)^{-\frac{a+1}{p}-\epsilon+i\xi}.
$$
Since $e^{-|\xi|\pi}\leq|(z+\epsilon i)^{i\xi}|\leq e^{|\xi|\pi}$, by \cite[Pag.10]{Hung2025} we have that
$$
\|f_{\epsilon,\xi}\|_{A^{p}_{{a}}}\sim\|f_{\epsilon,0}\|_{A^{p}_{{a}}}\sim \epsilon^{-1/p}<\infty,
$$
where the constants involved in the first comparison depend on $\xi$. With the same computations done in the proof of Theorem \ref{Theorem spectrum weighted Hardy}, we obtain that
\begin{align*}
&\frac{\|\mathcal{H}_{\phi_{\delta}}(f_{\epsilon,\xi})- \widehat{k_{\delta}}(\xi)f_{\epsilon,\xi} \|_{A^{p}_{{a}}} }{\|f_{\epsilon,\xi}\|_{A^{p}_{{a}}}}\\
&\qquad \qquad \leq C\int_{\delta}^{1/\delta}|\phi(t)| t^{\frac{a+1}{p}-1}\,dt\left[\epsilon\delta^{\delta^{-(2+\frac{a+1}{p})}}+\epsilon^{1/p}(\frac{1+a}{p}+\epsilon+|\xi|)\right]
\end{align*}
which goes to zero when $\epsilon$ tends to $0$.
More generally, let $\varepsilon>0$ small and pick $\delta>0$ such that $\|\mathcal{H}_{\phi}-\mathcal{H}_{\phi_{\delta}}\|_{A^{p}_{a} }<\varepsilon$ and $|\hat{k}(\xi)-\widehat{k_\delta}(\xi)|<\varepsilon$. Then 
\begin{align*}
&\lim_{\epsilon\to 0}\frac{\|\mathcal{H}_{\phi}f_{\epsilon,\xi}- \hat{k}(\xi)f_{\epsilon,\xi} \|_{A^{p}_{a}}} {\|f_{\epsilon,\xi}\|_{A^{p}_{a}}}\\
&\quad \leq \lim_{\epsilon\to0}\left(\frac{\|\mathcal{H}_{\phi_{\delta}}(f_{\epsilon,\xi})- \widehat{k_\delta}(\xi)f_{\epsilon,\xi}\|_{A^{p}_{a} }}{\|f_{\epsilon,\xi}\|_{A^{p}_{a} }}+\|\mathcal{H}_{\phi}-\mathcal{H}_{\phi_{\delta}}\|_{A^{p}_{a} }+ |\hat{k}(\xi)-\widehat{k_\delta}(\xi)|\right)\\
&\quad<2\varepsilon.
\end{align*}
Since $\varepsilon$ is arbitrary, this implies that $\hat{k}(\xi)$ is in the approximate point spectrum of $\mathcal{H}_{\phi}$, i.e 
$$
\sigma(\mathcal{H}_{\phi}, L^{p}_{a}(\mathbb{U}))=\sigma(\mathcal{H}_{\phi}, A^{p}_{a}(\mathbb{U}))=\overline{\hat{k}(\mathbb{R})}.
$$

\end{proof}

As already observed for the weighted Hardy spaces, the second part of the proof of Theorem \ref{Theorem spectrum Bergman} provides an original lower bound for the norm of the Hausdorff operator. Indeed
$$
\sup_{\xi \in {\mathbb{R}}} \left\vert \widehat{k_{a,p}}(\xi) \right\vert \leq\|\mathcal{H}_{\phi}\|_{A^{p}_{{a}}},
$$
where $\widehat{k_{a,p}}(\xi)=\int^{\infty}_{0}\phi(t)t^{\frac{a+1}{p}}\frac{dt}{t^{1+i\xi}}$. 

\medskip
%%%%%%%%%%%%%%%%%%%%%%%%%%%
\section{Cesaro Type operators: proof of Corollary \ref{last corollary}}

We recover and extend the results of \cite[Theorem 3.4]{Ballamoole2015} regarding the spectra of the Ces\'{a}ro-like operators. Let $\nu \in \mathbb{C}$ be such that $\mbox{Re}\nu > 0$. We recall that the Ces\'{a}ro-like operator is defined as
$$
C_{\nu}f(z)=\frac{1}{z^{\nu}}\int_{0}^{z}\zeta^{\nu-1}f(\zeta)\,d\zeta.
$$

\begin{proof}[Proof of Corollary \ref{last corollary}]
The Cesaro operator $\mathcal{C}_{\nu}$ is a specific Hausdorﬀ operator $\mathcal{H}_{\phi}$ associated with the kernel $\phi_{\mathcal{C}_{\nu}}(t)=\frac{1}{t^{\nu}} \chi_{[1,\infty)}(t)$. We show that  $\mathcal{C}_{\nu}$ satisfies the hypothesis of Theorems \ref{Theorem spectrum weighted Hardy} and \ref{Theorem spectrum Bergman}. Indeed let $1+a<p\cdot \mbox{Re } \nu$, then
\begin{align*}
\|C_{\nu}\|&=\int_{0}^{\infty}|\phi_{\mathcal{C}_{\nu}}(t)|t^{\frac{a+1}{p}-1}\,dt=\int_{1}^{\infty}t^{\frac{a+1}{p}-\mbox{Re}\nu-1}\,dt\\
&=\left[\frac{1}{\frac{a+1}{p}-\mbox{Re}\nu}t^{\frac{a+1}{p}-\mbox{Re }\nu}\right]_{1}^{\infty}=\frac{p}{p\mbox{Re}\nu -a-1}<\infty. 
\end{align*}
\begin{comment}
\begin{align*}
\int_{0}^{\infty}|\phi_{\mathcal{C}}(t)|t^{\frac{a+1}{p}-1}\,dt
&=\int_{1}^{\infty}t^{\frac{a+1}{p}-2}\,dt\\
&=\left[\frac{1}{\frac{a+1}{p}-1}t^{\frac{a+1}{p}-1}\right]_{1}^{\infty}\\
&=\frac{p}{p-a-1}<\infty. 
\end{align*}
\end{comment}
Thus $\mathcal{C}_{\nu}$ is bounded in the corresponding spaces. We observe that
\begin{align*}
\widehat{k^{\mathcal{C}_{\nu}}_{a,p}}(\xi)&=\int^{\infty}_{0}\phi_{
\mathcal{C}_{\nu}}(t)t^{\frac{a+1}{p}}\frac{dt}{t^{1+i\xi}}=\int_{1}^{\infty}t^{\frac{a+1}{p}-1-\nu-i\xi}\,dt\\
&=\int_{0}^{\infty}e^{(\frac{a+1}{p}-\nu-i\xi)t}\,dt=\frac{1}{i\xi-\frac{a+1}{p}+\nu}.
\end{align*}
\begin{comment}
 \begin{align*}
\widehat{k_{\mathcal{C},a,p}}(\xi)&=\int^{\infty}_{0}\phi_{
\mathcal{C}}(t)t^{\frac{a+1}{p}}\frac{dt}{t^{1+i\xi}}\\
&=\int_{1}^{\infty}t^{\frac{a+1}{p}-2+i\xi}\,dt\\
&=\int_{0}^{\infty}e^{(\frac{a+1}{p}-1-i\xi)t}\,dt\\
&=\frac{1}{i\xi-\frac{a+1}{p}+1}.
\end{align*}   
\end{comment}
Thus its spectrum is the closure of the image of $i\mathbb{R}$ under the linear fractional map $(z-\frac{a+1}{p}+\nu)^{-1}$. The description of the essential spectrum in each case follows by the discussion of \cite[Remark 4.1]{OlivaMaza2023}   
\end{proof}

\medskip
\section{Hausdorff operators defined by measures}

Let $M(\mathbb{R}_+)$, $M(\mathbb{R})$ denote the space of positive, regular Borel measures respectively on the semi-axis $(0, \infty)$ and on the real line $\mathbb{R}$. For $\mu \in M(\mathbb{R}_+)$, we define the Hausdorff operator $\mathcal{H}_{\mu}$ on the weighted Hardy space $H^{p}_{|\cdot|^{a}}(\mathbb{U})$ as
\[
\mathcal{H}_{\mu}f(z) = \int_{0}^{\infty} f\left(\frac{z}{t}\right) \frac{d\mu(t)}{t},
\]
and similarly for $f \in L^{p}(\mathbb{R}, |\cdot|^{a})$ as
\[
H_{\mu}f(x) = \int_{0}^{\infty} f\left(\frac{x}{t}\right) \frac{d\mu(t)}{t}.
\]
Note that for these operators to be well-defined, $\mu$ cannot have a point mass at the origin. Following \cite[Lemma 4.4 (ii)]{Hung2025}, we observe that
\begin{align*}
\left|\mathcal{H}_{\mu}f(z)\right| &\leq \int_{0}^{\infty} \left|f\left(\frac{z}{t}\right)\right| \frac{d|\mu|(t)}{t} \\ 
&\leq w_{a}(B(x,y))^{-1/p} \|f\|_{H^{p}_{|\cdot|^{a}}} \int_{0}^{\infty} t^{\frac{a+1}{p}-1} d|\mu|(t),
\end{align*}
where $|\mu|$ denotes the total variation measure. This proves that $\mathcal{H}_\mu$ is well defined on $\mathbb{U}$. Consequently, the condition 
\begin{equation}\label{Boundedness mu}
 \int_{0}^{\infty} t^{\frac{a+1}{p}-1} d|\mu|(t) < \infty  
\end{equation}
is sufficient for the boundedness of $\mathcal{H}_{\mu}$. Indeed, by Minkowski's inequality, we have that 
\[
\|\mathcal{H}_\mu\|_{H^{p}_{|\cdot|^{a}}} \leq \int_{0}^{\infty} t^{\frac{a+1}{p}-1} d|\mu|(t) \quad \text{and} \quad \|H_\mu\|_{L^{p}(\mathbb{R}, |\cdot|^{a})} \leq \int_{0}^{\infty} t^{\frac{a+1}{p}-1} d|\mu|(t).
\]
Analogous to the case of absolutely continuous measures, one may verify that
\begin{equation}\label{apr spe}
 \overline{\widehat{k_{a,p}}(\mathbb{R})} \subseteq \sigma(\mathcal{H}_{\mu}, H^{p}_{|\cdot|^{a}}(\mathbb{U})) \quad \text{and} \quad \sup_{\xi \in \mathbb{R}} \left| \widehat{k_{a,p}}(\xi) \right| \leq \|\mathcal{H}_{\mu}\|_{H^{p}_{|\cdot|^{a}}},
\end{equation}
where $\widehat{k_{a,p}}(\xi) = \int^{\infty}_{0} t^{\frac{a+1}{p}} \frac{d\mu(t)}{t^{1+i\xi}}$.

\medskip
As established in Section 2, the spectral analysis of the Hausdorff operator $H_\mu$ on the power-weighted space $L^p(\mathbb{R}, |\cdot|^a)$ reduces to the study of a convolution operator
\begin{equation}
    \sigma(H_\mu, L^{p}(\mathbb{R}, |\cdot|^{a})) = \sigma(T_\nu, L^{p}(\mathbb{R})),
\end{equation}
where $T_{\nu}f = \nu * f$ is the convolution operator on $L^p(\mathbb{R})$ associated with the measure $\nu$ defined by
\begin{equation}\label{pushback}
d\nu(s) = e^{(\frac{a+1}{p}-1)s}dL^{*}[\mu](s),
\end{equation}
and $L^{*}\mu$
is the push forward measure associated with the function $L(s)=\ln(s)$.

Consequently, characterizing the spectrum of the Hausdorff operator is equivalent to analyzing the spectrum of its corresponding convolution measure.

For $\nu, \lambda \in M(\mathbb{R})$, the convolution $\nu *\lambda$ is defined as the unique measure such that for any Borel set $A \subset \mathbb{R}$
\[
(\nu *\lambda)(A)=\int_{\mathbb{R}}\int_{\mathbb{R}}\chi_{A}(x+y)\, d\nu(x)\, d\lambda(y).
\]
Since the Dirac measure $\delta_0$ acts as the identity, that is $\delta_0 * \nu = \nu$, we can consider $T_\nu$ acting on the measure algebra $M(\mathbb{R})$. Following Zafran \cite{Zafran1973}, a measure $\nu \in M(\mathbb{R})$ is said to possess a {natural spectrum} if 
\[
\sigma(T_\nu, M(\mathbb{R})) =\hat{\nu}(\mathbb{R}) \cup \{0\}.
\]
We denote the set of such measures as $\mathcal{N}(\mathbb{R})$. While $\mathcal{N}(\mathbb{R})$ contains all absolutely continuous and discrete measures, the {Wiener-Pitt phenomenon} \cite{Wiener1938} establishes that this property is not universal. Specifically, there exist measures in $M_{0}(\mathbb{R})$---the closed ideal of measures whose Fourier-Stieltjes transforms vanish at infinity---which do not possess a natural spectrum.

Suppose that $\nu \in \mathcal{N}(\mathbb{R})$. For any $\lambda \in \mathbb{C} \setminus \left( \hat{\nu}({\mathbb{R}}) \cup\{0\}\right)$, there exists an inverse measure $\nu_\lambda \in M(\mathbb{R})$ such that $(\lambda\delta_0 - \nu) * \nu_\lambda = \delta_0$. This implies that $\lambda I - T_\nu$ is invertible with inverse $T_{\nu_\lambda}$, and consequently $\sigma(T_\nu, L^p(\mathbb{R})) \subseteq \hat{\nu}(\mathbb{R}) \cup \{0\}$.

By lifting this back to $L^p(\mathbb{R}_+, |\cdot|^a)$, we define the operator $H_{\nu_\lambda^*}$ via the measure 
\[
d\nu_{\lambda}^*(s) = s^{1-\frac{a+1}{p}}dE^{*}[\nu_{\lambda}](s),
\]
where $E^{*}{\nu_\lambda}$ is the push forward measure associated with the function $E(s)=e^s$.
The total variation satisfies
\[
\int_{0}^{\infty} t^{\frac{a+1}{p}-1} d|\nu_{\lambda}^*|(t) = \int_{\mathbb{R}} d|\nu_{\lambda}|(t) = \|\nu_\lambda\| < \infty,
\]
ensuring that $H_{\nu_\lambda^*}$ and $\mathcal{H}_{\nu_\lambda^*}$ are bounded and they act as the inverses of $\lambda I - H_\mu$ and $\lambda I - \mathcal{H}_\mu$. Following the proof of Theorem \ref{Theorem spectrum weighted Hardy}, we obtain
\[
\sigma(\mathcal{H}_\mu,H^{p}_{|\cdot|^{a}}(\mathbb{U})) \subseteq \hat{\nu}(\mathbb{R}) \cup \{0\} \quad \text{and} \quad \sigma(H_\mu,L^{p}(\mathbb{R},|\cdot|^{a})) \subseteq \hat{\nu}(\mathbb{R})\cup \{0\}.
\]
Applying \eqref{apr spe} and \eqref{pushback}, it follows also that
\[
\sigma(\mathcal{H}_\mu,H^{p}_{|\cdot|^{a}}(\mathbb{U})) = \widehat{k_{a,p}}(\mathbb{R})\cup\{0\}.
\]

The situation is more delicate when $\nu$ does not have a natural spectrum. While the test functions from Theorem \ref{Theorem spectrum weighted Hardy} show that $\hat{\nu}(\mathbb{R})$ is contained in the approximate point spectrum, the identity $\sigma(T_\nu, M(\mathbb{R})) = \hat{\nu}(\mathbb{R}) \cup \{0\}$ may fail, and a full characterization of $\sigma(H_\mu, L^{p}(\mathbb{R}, |\cdot|^a))$ remains to be understood.

However, for specific classes of measures, progress is still possible. If for example $\nu \in M_0(\mathbb{R})$, then, for $1 < p < \infty$
\begin{equation}\label{last}
  \sigma(T_\nu, L^p(\mathbb{R})) = \hat{\nu}(\mathbb{R}) \cup \{0\}.  
\end{equation}
This and \eqref{apr spe} allows us to conclude that $\sigma(H_{\mu}, L^{p}_{|\cdot|^{a}}(\mathbb{R})) = \overline{\widehat{k_{a,p}}(\mathbb{R})}$. Furthermore, as in \cite[Theorem 1.5, (ii)]{Hung2025}, $H_{\mu}$ commutes with the Cauchy-Szeg\H{o} projection $P: L^{p}(\mathbb{R}, |\cdot|^{a}) \to H^{p}_{|\cdot|^{a}}(\mathbb{U})$ for $-1 < a < p-1$. By applying Proposition \ref{subspace theorem} and the relation \eqref{last}, we conclude that if $H_{\mu}$ commutes with $P$ and the measure 
\[
e^{(\frac{a+1}{p}-1)s}dL^{*}[\mu](s) \in M_0(\mathbb{R}),
\]
then, for $-1 <a< p -1$,
\[
\sigma(\mathcal{H}_{\mu}, H^{p}_{|\cdot|^{a}}(\mathbb{R})) = \widehat{k_{a,p}}(\mathbb{R})\cup\{0\}.
\]

\section*{Declarations}
No data was used for the research described in the article. The authors assume full responsibility for the contents of this article.

\bibliographystyle{plain}

\bibliography{ProJectA_B}

\medskip

\end{document}